\documentclass{amsart}
\usepackage[all]{xy}
\SelectTips{cm}{}
\usepackage{amsmath}
\usepackage{amssymb}
\usepackage{amscd}
\usepackage{amsthm}
\usepackage{amsfonts}
\usepackage{amsxtra}

\theoremstyle{plain}

\newtheorem{Thm}{Theorem}[section]
\newtheorem{Lem}[Thm]{Lemma}
\newtheorem{Pro}[Thm]{Proposition}
\newtheorem{Cor}[Thm]{Corollary}

\newtheorem{bigthm}{Theorem}

\newtheorem{bigadd}[bigthm]{Addendum}

\theoremstyle{definition}
\newtheorem{Def}[Thm]{Definition}
\newtheorem{Exa}[Thm]{Example}

\newtheorem{Rem}[Thm]{Remark}

\newtheorem*{Rem-intro}{Remark}
\DeclareMathOperator{\Ad}{Ad}
\DeclareMathOperator{\Pad}{Pad}

\newcommand{\ad}{{\mathrm{ad}}}

\newcommand{\Dirlim}{\varinjlim}
\newcommand{\Invlim}{\varprojlim}

\newcommand{\ZZ}{{\mathbb{Z}}}
\newcommand{\QQ}{{\mathbb{Q}}}

\newcommand{\CC}{{\mathbb{C}}}
\newcommand{\GG}{{\mathcal{G}}}

\newcommand{\idc}{_{\circ}}

\def\:{\colon\!}

\begin{document}

\title[Continuous trace $C^*$-algebras and gauge groups]
{ Continuous trace $C^\ast$-algebras, gauge groups and rationalization }
\author[Klein]{John R. Klein}
\address{Department of Mathematics,
     Wayne State University,
     Detroit MI 48202}
\email{klein@math.wayne.edu}
\author[Schochet]{Claude L.~Schochet}
\address{Department of Mathematics,
     Wayne State University,
     Detroit MI 48202}
\email{claude@math.wayne.edu}
\author[Smith]{Samuel~B.~Smith}
\address{Department of Mathematics,
     Saint Joseph's University,
     Philadelphia PA 19131}
\email{smith@sju.edu}
\thanks{The first   author is partially supported by the National Science Foundation}
\keywords{continuous trace $C^*$-algebra, section space,
gauge group, projective gauge group,
 rational $H$-space, topological group, localization}
\subjclass[2000]{46J05, 46L85, 55P62, 54C35, 55P15, 55P45}
\begin{abstract}
Let $\zeta$   be  an   $n$-dimensional complex  matrix bundle   over a
compact metric space
$X$ and let $A_\zeta$ denote the $C^*$-algebra of sections of this
bundle.
We determine the rational homotopy type as an $H$-space
of $UA_\zeta$, the group of unitaries  of $A_\zeta$.
The answer turns out to be
independent of the bundle $\zeta $ and
depends only upon $n$ and the rational cohomology of $X$.
We prove analogous
results for the gauge group and the projective gauge group of a
  principal bundle over a compact metric space $X$.

\end{abstract}
\maketitle
\setcounter{tocdepth}{1}
\tableofcontents

 \section{Introduction}

We analyze the rational homotopy theory
of certain topological
groups arising  from bundles   over a compact metric space $X$.  Our results are  motivated by the following situation. Let 
$U_n$ be the unitary group of $n\times n$ matrices, and let $PU_n$ be
the group given by taking the quotient of $U_n$ with its center.
Let $\zeta \: T \to X$ be a principal $PU_n$-bundle over $X$, let $PU_n$ act on
 $M_n = M_n(\CC)$ by conjugation and let
$$
 T \times_{PU_n} M_n  \to X
$$
be the associated $n$-dimensional complex matrix bundle. Define
 $A_\zeta$ to be the set of continuous sections of the latter.
These sections
have natural pointwise addition, multiplication, and $*$-operations
and give $A_\zeta $ the structure of a  unital $C^*$-algebra. The algebra
$A_\zeta$ is called an {\emph{$n$-homogeneous $C^*$-algebra}}
and  is the most general
 unital  {\emph{continuous trace}} $C^*$-algebra as studied, for
instance, in  the book of Raeburn and Williams \cite{RW}.
Let $UA_\zeta $ denote the topological group of unitaries of $A_\zeta $.
Our first main result describes the rational homotopy type of $UA_\zeta.$

Recall
that, from the point of view of homotopy theory, the simplest
groups
are  the Eilenberg-Mac\,Lane spaces $K(\pi , n)$  with multiplication
given by
$$
\begin{CD}
K(\pi , n) \times K(\pi , n) \simeq K(\pi \times \pi , n)
@> K(\text{\rm multiply}) >>  K(\pi , n).
\end{CD}
$$
Here $\pi$ is an abelian group and the space $K(\pi, n)$ satisfies $\pi_i(K(\pi, n)) = \pi$ for $i =n$ and $\pi_i(K(\pi, n)) = 0$ for 
$i \neq n.$   Only some of the constructions of a $K(\pi, n)$
yield a bona fide topological group, but all yield an $H$-space; that is,
a space with continuous binary operation and two sided unit. However, this
discrepancy is not hard to rectify: up to homotopy
all of these $H$-space structures on Eilenberg-Mac\,Lane spaces
lift to topological group structures in the
sense that there is a topological group $G$ and a homotopy equivalence to the given
$K(\pi, n)$ which preserves the multiplication up to homotopy.

In fact, the $H$-space structure on a given
Eilenberg-Mac\,Lane
space is unique up to multiplicative equivalence and  is
{\em homotopy commutative}.
A  product $\prod_{j \geq 1} K(\pi_j, j)$ of Eilenberg-Mac\,Lane spaces
also has a  preferred  $H$-space structure
given by the product of the structures on the factors.  This structure, which we refer to as the
{\em standard multiplication,} is also homotopy commutative.  However, in  this case
this structure may not be unique (See \cite{Cur}).

Given a simply connected CW space $X$,
Sullivan constructed a rationalization
map $X \to X_{\QQ}$ which has the property that the associated homomorphism on
the higher homotopy groups is given by tensoring with the rational numbers
(\cite{Su}; rationalization is a special case of a more general construction, localization,
that can be made for any set of primes).
Later, this theory was extended to include nilpotent spaces, i.e., spaces with non-trivial nilpotent fundamental group $\pi$ 
having the property that the higher homotopy groups are nilpotent modules
over $\pi$ (\cite{HMR}, \cite{BK}).

It is well-known that topological groups are nilpotent spaces, so one can
consider the rationalization map $G \to G_{\QQ}$ for connected topological groups $G$ (whose
underlying space is a CW complex). Since localization commutes with finite products up to homotopy,
it follows that $G_{\QQ}$ has the structure of an $H$-space, and furthermore, the rationalization
map is an $H$-map, i.e., it preserves multiplications up to homotopy.
This motivates the following: let us
call two $H$-spaces $X$ and $Y$
{\it rationally $H$-equivalent}
if there is a homotopy equivalence $X_\QQ \to Y_\QQ$ which is a
map of $H$-spaces.

To state our calculation of the rational homotopy groups of $UA_\zeta$,
we introduce  some notation.   Given   $\mathbb Z$-graded vector spaces $V$ and
$W$, we grade the tensor product $V\otimes W$  by declaring that $v\otimes w$ has
 degree $|v| + |w|$. Here $|v|$ denotes the degree of the element $v \in V.$ Let $V\widetilde{\otimes} W$ be the effect of considering
only tensors with non-negative grading.

Given   elements $x_1, \ldots, x_n$ each of homogeneous degree,
write $\QQ(x_1, \ldots,  x_n )$ for the graded vector space  with basis $x_1, \ldots, x_n$.
Given a topological group $G$,
write $G\idc$ for the path component of the identity in $G$.
Let  $\Check{H}^*(X; \QQ)$ denote  the \v{C}ech cohomology of a space $X$
with rational coefficients {\emph{graded nonpositively}} so that $x \in \Check{H}^n(X; \QQ)$ has degree $-n$. 

\begin{bigthm} \label{bigthm:mainA}
Let $\zeta$   be  a principal $PU_n$
 bundle over  a compact metric space $X$.  Let
$A_\zeta$ be the associated $C^*$-algebra, and let $UA_\zeta$ be its group of unitaries.
Then the rationalization of $(UA_\zeta)\idc$ is rationally $H$-equivalent to a product of
rational Eilenberg-Mac\,Lane spaces  with the standard multiplication,
with degrees and dimensions corresponding to an isomorphism of graded vector spaces
$$
\pi_*\left((UA_\zeta)\idc \right) \otimes \QQ \,\, \cong \,\, \Check{H}^*(X; \QQ) \,
\widetilde{\otimes}  \,  \QQ(s_1, \ldots, s_n) \, ,
$$
where the basis element
$s_i$ has degree $2i-1$.

\end{bigthm}

\medskip

Theorem \ref{bigthm:mainA} is a special case of more general
calculations of the rational homotopy  theory  of gauge groups
which we now describe.
Write $F(X, Y)$ for the (function) space of all continuous maps from
$X$ to $Y$.
When $G$ is a topological group, the space $F(X, G)$ is one also with
multiplication of functions taken pointwise.
In this case, the identity component  $F(X,G)_\circ $ is the space of freely nullhomotopic maps.

\begin{bigthm}\label{bigthm:mainB}
Let $X$ be a compact metric space and
let  $G$ be a connected topological group having the homotopy type of a finite CW complex. Then
$$
\pi_*(F(X, G)\idc) \otimes \QQ \,\, \cong\,\,  \Check{H}^*(X; \QQ)
\, \widetilde{\otimes} \, \left( \pi_*(G) \otimes \QQ \right)\,  .
$$
 Furthermore, $F(X, G)\idc$ is rationally $H$-equivalent
 to a product of Eilenberg-Mac\,Lane spaces with
the standard multiplication, with degrees and dimensions corresponding
to the displayed isomorphism.
\end{bigthm}

 When $X$ is a finite complex, Theorem \ref{bigthm:mainB} is a consequence
of  results of Thom \cite{Thom}
and a basic localization result for components of $F(X,Y)$ due to Hilton, Mislin and Roitberg
\cite{HMR}.   The result for $X$ finite in this case is described  in  \cite[\S4]{LPSS}.                   Our advance here is the 
extension
of this result to the case when $X$  is compact metric.  We   deduce  Theorem \ref{bigthm:mainB} from an  extension of the  
Hilton-Mislin-Roitberg
  result to the case $X$ compact metric (Theorem \ref{thm:HMR extended}).

\begin{bigadd} \label{addB} In Theorem \ref{bigthm:mainB},
the calculation of  rational homotopy groups holds for any connected, group-like H-space $G$.
Furthermore, if $G$ is rationally homotopy commutative, then
$F(X,G)$ is rationally $H$-equivalent to a product of Eilenberg-Mac\,Lane spaces with
the standard multiplication.
\end{bigadd}

The main results of this paper concern extending Theorem \ref{bigthm:mainB}  to  spaces of sections of certain bundles. Let 
$G$ be a topological group and let
$$
 \zeta\: T\to  X
$$
be a principal $G$-bundle.
Following \cite[\S2]{AB}, we form the associated {\it adjoint bundle}
$$
\Ad(\zeta) \:   T \times_G G^{\ad} \to X
$$
where $G$ acts on $G^{\ad} = G$ by conjugation.
The {\em gauge group}  $\GG(\zeta)$ of $\zeta$ is the space of sections
of $\Ad(\zeta)$, with group structure defined by pointwise multiplication
of sections. Alternatively, $\GG(\zeta)$ is
the group of $G$-equivariant bundle automorphisms of $\zeta$ that cover the
identity map of $X$.

\begin{bigthm}
\label{bigthm:mainC} Let $G$ be a connected topological group having the homotopy
type of a finite CW complex. Let
 $\zeta$ be a principal $G$-bundle over a compact
metric space $X$.  Then there is a
  rational $H$-equivalence
$$\GG(\zeta)\idc \,\, \simeq_\QQ \,\, F(X, G)\idc\, .$$ Consequently, $\GG(\zeta)\idc$ is rationally
homotopy commutative with rational homotopy groups given by
the isomorphism appearing in Theorem \ref{bigthm:mainB}.
\end{bigthm}
Again, when $X$ is a finite CW complex this result admits a direct proof.
In this case,   a result of Gottlieb   gives  a multiplicative equivalence
\[
\GG(\zeta) \simeq \Omega_{h_\zeta} F(X, BG),
\]
 where the right side denotes the loop space
of $F(X,BG)$ based at $h_\zeta \: X \to BG$, the
classifying map for $\zeta$ (see Corollary \ref{cor:LG}, \cite[th.\ 1]{Got} and  \cite[prop.\ 2.4]{AB}).  The equivalence  
$\GG(\zeta) \simeq_\QQ F(X, G)$ then follows from the Hilton-Mislin-Roitberg localization result for function spaces  
mentioned above and basic rational homotopy theory. (See Theorem \ref{thm:mainC} below.)  The result in this case was  
recently, independently obtained  by F\'elix and Oprea at the level of rational homotopy groups \cite[th.\ 3.1]{FO}.   Another 
related result here  is due to Crabb and Sutherland, who prove the fibrewise rationalization of the  bundle $\Ad(\zeta_G)$  is 
fibre homotopically trivial, where $\zeta_G$ is the universal $G$-bundle \cite[prop.\ 2.2]{CS}.

The following shows that the homotopy finiteness assumption on $G$ in Theorem \ref{bigthm:mainC} can sometimes be  
dispensed with.

\begin{bigadd} \label{addC} Assume $G$ is a topological group
such that $BG$ has the rational homotopy type of a group-like $H$-space.
Then the conclusion of Theorem  \ref{bigthm:mainC} holds for such $G$.
\end{bigadd}

For example,  if $G$ is a connected topological group satisfying rational Bott periodicity, then
$BG$ has the rational homotopy type of a group-like $H$-space.
\medskip

Let  $$PG \,\, = \,\, G/Z(G)$$ denote the
projectivization of $G$; i.e., the quotient of $G$ by its center. As the center
acts trivially on $G^{\ad}$, one obtains an action of $PG$ on $G^{\ad}$.
Given a principal $PG$-bundle
\[
\zeta\:T \to  X,
\]
form the associated {\em projective adjoint} bundle
$$
\Pad(\zeta )\:  T\times_{PG} G^{\ad} \to X.
$$
Define $\mathcal{P}(\zeta)$ to be the topological group of sections of the bundle $\Pad(\zeta)$ with pointwise multiplication 
again induced by $G^{\ad}.$  We
call $\mathcal{P}(\zeta)$ the {\em projective gauge group} of $\zeta$.
In Example \ref{thm:UA proj gauge} below, we observe that
$UA_\zeta \cong \mathcal{P}(\zeta) $
corresponds to the projective adjoint bundle of a principal $PU_n$-bundle.  Theorem \ref{bigthm:mainA}  is thus a special 
case of the following result.

\begin{bigthm}
\label{bigthm:mainD}
Let $G$ be a compact connected Lie group. Let $\zeta$ be a principal $PG$-bundle over a compact
metric space $X$.  Then there is  a
  rational $H$-equivalence
$$\mathcal{P}(\zeta)\idc \,\, \simeq_\QQ \,\, F(X, G)\idc\, .$$
Thus $\mathcal{P}(\zeta)\idc$ is rationally
homotopy commutative with rational homotopy groups again given by  the
isomorphism appearing in Theorem \ref{bigthm:mainB}.
\end{bigthm}

\begin{Rem} \label{about_C*}  Suppose that $C$ is a separable $C^\ast$-algebra. Then
its unitary group $UC$ (with the usual modification for non-unital $C$) has the homotopy type of a countable CW complex. 
Thus so too does $U_\infty C = \Dirlim U_nC$, and   the latter
 is an infinite loop space, by the Bott Periodicity Theorem of R.\ Wood \cite{Wood}. Thus $U_\infty C$ satisfies the conditions 
 on $G$ in Addenda \ref{addB} and \ref{addC}.
The same is  true for $UC$ itself if $C$ is stable.  So our
results also apply to $C^\ast $-algebras constructed similarly to $A_\zeta $ but where the initial fibre $M_n(\CC )$ is replaced 
by  an
appropriate $C^\ast$-algebra $C$.   We develop these ideas in a subsequent paper.
\end{Rem}

The paper is organized as follows.  In Section \ref{sec:conventions},  we establish our
basic conventions for  spaces, groups and bundles.
 In Section  \ref{section_spaces}, we prove various foundational properties of section spaces.
 In Section \ref{sec:groups}, we discuss  the rationalization of topological groups and
the obstruction to homotopy commutativity.
 In Section \ref{sec:prelims}, we prove preliminary versions
of the main theorems  for $X$ a finite CW complex, as mentioned above.

  By a  result of  Eilenberg and Steenrod  \cite{ES},    a compact metric space $X$ may be expressed as the  inverse limit 
  $\Invlim_{j} X_{j}$ of finite CW complexes.
In Section \ref{sec:functions}, we
use this result and the classical works of Dowker \cite{Dow} and Spanier \cite{Span}  to
identify the homotopy groups  of the
function space $F(X,Y)$ in terms of the homotopy groups of the 
approximating function spaces $F(X_j,Y)$.
This result is subsequently extended to section spaces.
As a consequence, in Section \ref{sec:HMR}, we extend the basic localization
result   of Hilton-Mislin-Roitberg \cite[th.\ II.3.11]{HMR}
for function spaces from the case $X$ finite CW to the case $X$ compact, metric
provided  the function space component is a nilpotent space (Theorem \ref{thm:HMR extended}).
In Section \ref{sec:sections} we deduce Theorems \ref{bigthm:mainA}-\ref{bigthm:mainD} by combining the finite complex 
case with the results of Section \ref{sec:functions}.
\medskip

{\flushleft \bf Acknowledgments.}
We thank  Daniel Isaksen,
Gregory Lupton, and J. Peter May  for many helpful discussions. We are especially grateful to  N.~Christopher Phillips
 for vital assistance given to us. This paper is based in many ways upon our joint work \cite{LPSS}.
\newline \ \ \newline


\section{Conventions} \label{sec:conventions}

 This paper brings together results from classical algebraic topology, which is most at home in the
 category of CW complexes, and functional analysis, which is most at home in the category of compact
 metric spaces. Many of our technical results deal with extending classical algebraic topology results from finite complexes to 
 compact metric spaces via limit arguments.

 We work in the category of compactly generated Hausdorff spaces.
 Whenever basepoints are required we assume that they are non-degenerate; that is, we assume
 that the inclusion of the basepoint into the space is a cofibration.
 If the space is a topological group then we take the identity of the group to be the basepoint.
Following the discussion in \cite[pp.\ 20-21]{Wh}, we give the function space  $F(X, Y)$   the topology obtained by first taking 
the compact-open topology and then replacing this with the induced compactly generated topology. In particular, because we 
are retopologizing products, by
a {\it topological group} we mean a topological group object in the category of compactly generated
Hausdorff spaces.

Suppose $A\subset X$ is a subspace. Fixing a map $g\: A \to Y$, we let
  $F(X, Y; g)$ denote  the subspace of those maps $f\: X \to Y$ such that $f$ coincides with
  $g$ on $A$. In particular, when $A$ is a point, $X$ and $Y$ obtain the structure of based
  spaces and $F(X,Y;g)$ in this case is just the space of based maps.
  If $f\in F(X,Y;g)$ is a choice of basepoint, we let $F(X,Y;g)_{(f)}$ denote
  the path component of $f$.

 An inclusion $A\subset X$ of spaces is a {\it cofibration} if   it satisfies
the homotopy extension property. A {\it Hurewicz fibration} is a map
$ p\: E \to B$
satisfying the homotopy lifting problem for all (compactly generated) spaces.
A {\it map of fibrations} $E\to E'$ over $B$ is a map of spaces which commutes with
projection to $B$.
One says that $p$ is {\it fibre homotopy trivial} if there is a space $F$
and a map $q\:E\to F$ such that $(p,q)\: E \to B \times F$ is a homotopy equivalence.
Given a map $f \: Y \to B$ we write $f^*(p) \: Y\times_B E \to Y$ for the pullback
fibration (i.e., the fibre product).

An {\it $H$-space structure} on a based space $X$ is a map
$m\:X\times X\to X$ whose restriction to $X\times \ast$ and $\ast \times X$
is homotopic to the identity as based maps, where $\ast\in X$ is the basepoint.
If an $H$-space structure on $X$ is understood, we call $X$ an {\it $H$-space.}
One says that $X$ is {\it homotopy associative} if the maps
$m \circ (m\times \text{id})$ and $m  \circ (\text{id}\times m)$ are homotopic.
A {\it homotopy inverse} for $X$ is a map $\iota\: X \to X$ such that the composites
$m \circ (\iota \times \text{id})$ and $m \circ (\text{id} \times \iota)$ are homotopic
to the identity. If $X$ comes equipped with a homotopy associative multiplication
and a homotopy inverse, then $X$ is said to be {\it group-like.} If $X$
is group-like then the set of path components $\pi_0(X)$ acquires a group structure.

\subsection*{Nilpotent Spaces}
If $(X,\ast)$ is a based space
then its higher homotopy groups $\pi_n(X;\ast)$ come equipped with
an action of the fundamental group $\pi = \pi_1(X,\ast)$.
If $X$ is also a connected CW complex, then we say that $X$ is
 {\it nilpotent} if $\pi$ is a nilpotent group and also the action
 of $\pi$ on the higher homotopy groups is nilpotent. The latter condition
 is equivalent to the statement that each $\pi_n(X;\ast)$ possesses
a finite filtration of
$\pi$-modules $M_n(i) \subset M_n(i+1)\subset \cdots$ such that the action on the associated graded
 $M_n(i+1)/M_n(i)$ is trivial.
 More generally, if $X$ is any based connected space, then we will call $X$  {\it nilpotent}
 if $X$ has the homotopy type of a nilpotent CW complex.
 Topological groups having the homotopy type of a connected CW complex are nilpotent,
 since the action of $\pi_1$ in this case is trivial.

\subsection*{Rationalization}

A finitely generated
nilpotent group $K$ admits a rationalization, which is
a natural homomorphism
$ K \longrightarrow K_\QQ$ (\cite[\S2]{HMR}).
The group $K_\QQ$ has the property that the self map
$x\mapsto x^n$ is a bijection for all integers $n \ge 1$ (i.e., $K_\QQ$ is
uniquely divisible).
Furthermore,  $K_\QQ$ is the smallest group having this property
in the sense that any homomorphism from $K$ to a group with this property
uniquely factors through $K_\QQ$. When
$K$ is abelian, there is a natural isomorphism $K_{\QQ} \cong K \otimes \QQ.$

A connected based nilpotent space $X$ admits a {\it rationalization.}  This is a
nilpotent space $X_\QQ$ with rational homotopy (and homology) groups, together with
a natural map $X \to X_\QQ$
 and a natural map ${\ell}_X \colon X \to X_\QQ$ inducing rationalization on homotopy groups
 \cite[thms.\ 3A, 3B]{HMR}.  Again, there is a universal property: if $Y$ is a rational
 space (i.e., a nilpotent space whose homotopy groups are rational), and $f\:X \to Y$ is a map, then
 one has a commutative diagram
 $$
\xymatrix{
 X \ar[r] \ar[d]_f & X_\QQ\ar[d]^{f_{\QQ}}   \\
 Y \ar[r]_\simeq & \, Y_\QQ \,  ,
}
 $$
 where the bottom map is a homotopy equivalence since $Y$ is rational. Consequently,
$f$ factors uniquely up to homotopy through $X_\QQ$, and
in particular  the rationalization of $X$ is uniquely determined up to homotopy
equivalence. More generally, we call a map $f\:X \to Y$ a {\it rationalization} if $Y$ is rational
and the induced map $f_\QQ\: X_\QQ \to Y_\QQ$ is a homotopy equivalence. This is equivalent
to demanding $f^*\:[Y,Z] \to [X,Z]$ be an isomorphism of sets for all
rational nilpotent spaces $Z$.

\section{Section spaces} \label{section_spaces}

Suppose one is given a lifting problem, i.e.,
 a diagram of spaces
$$
\xymatrix{
A \ar[r]^g \ar[d]_{\cap} & E \ar[d]^p \\
X \ar[r]_f \ar@{..>}[ur] & B
}
$$
such that $A \subset X$ is a cofibration and $E\to B$ is a fibration.
Let us denote this lifting problem by ${\mathcal D}$.

Let
$$
\Gamma(\mathcal D)
$$
be the space of solutions to the lifting problem, i.e.,
the space of maps $X \to E$ making the diagram commute. When
$f$ is the identity map and $A$ is trivial, then one obtains the
space of sections of $p$ (it will be denoted by $\Gamma(p)$ in this
instance).

\begin{Pro} \label{Pro:pullback section sequence}
Let $\mathcal D$ be the lifting problem above. Then one has a fibration
$$
\begin{CD}
F(X,E;g) @> p_* >> F(X,B;p\circ g)
\end{CD}
$$
whose fibre at $f$ is given by $\Gamma(\mathcal D)$.
\end{Pro}

\begin{proof} Here $p_*$ is given by mapping a function $a\:X \to E$ to $p\circ a\: X \to B$.
The map $p_*$ is a fibration by the exponential law. The fibre over $f$ is clearly
$\Gamma(\mathcal D)$.
\end{proof}

\subsection*{CW structure}

\begin{Pro}\label{Pro:milnor}
With respect to the diagram ${\mathcal D}$ above, assume that
 $X$ is a compact metric space and suppose $E$ and $B$ have the homotopy type of CW complexes.
Then the section space $\Gamma(\mathcal D)$ has the homotopy type of a CW complex.
\end{Pro}

\begin{proof} Restriction $f\mapsto f_{|A}$ defines a fibration
$$
F(X,E) \to F(A,E)
$$
in which both the domain and codomain have the homotopy type of a CW complex
(by \cite[{th.\ 1}]{Mil}). Apply \cite[{prop.\ 3}]{Sch"o} (or \cite[lem. 2.4]{PJK})
to deduce that the fibre $F(X,E;g)$ of this fibration has the
 homotopy type of a CW complex. Repeating this argument
 with the fibration of Proposition \ref{Pro:pullback section sequence} completes the proof.
\end{proof}

\begin{Cor} Let $X$ be a compact metric space and $G$ a topological group
having the homotopy type of a  CW complex.
Let $\zeta$ be a principal $G$-bundle (respectively, principal $PG$-bundle) over $X$.
Then $\GG(\zeta)$ (respectively, $\mathcal P(\zeta)$) has the homotopy type
of a CW complex.
\end{Cor}

\begin{proof}  We only prove the case of the adjoint bundle as the other case
is proved similarly. The fibre bundle $\zeta$
is classified by a map $f\:X \to BG$ by pulling back the universal
principal $G$-bundle $EG \to BG$ along $f$. The space of solutions of the lifting problem
$$
\xymatrix{
 & EG\times_G G^{\ad} \ar[d] \\
X \ar[r]_f \ar@{..>}[ur] & BG
}
$$
coincides with the section space $\Gamma(\Ad(\zeta))$. Furthermore $EG\times_G G^{\ad}$
and $BG$ have the homotopy type of CW complexes because $G$ does. The proof is
completed by applying Proposition \ref{Pro:milnor}.
\end{proof}

\begin{Exa}\label{thm:UAisCW} Suppose that $A$ is a  (separable) unital Banach algebra.
Then the group of invertibles
$GL(A)$ has the homotopy type of a (countable)  CW complex.  If $A$ is a (separable) unital $C^*$-algebra then the group of 
unitaries $UA$ has the homotopy type of a (countable) CW complex.

Here is a proof.
As $UA$ is a deformation retraction of $GL(A)$, they both have the same homotopy type.
 The group $GL(A)$ is an open subset of a   Banach space, and any such open set
 has the homotopy type of a CW complex (cf.\ \cite[cor.\ IV.5.5]{LW}). If $A$ is separable then the  open
 covering involved in the proof of \cite[prop.\ IV.5.4]{LW}
  may be taken to be countable and then
 an obvious modification of \cite[IV.5.5]{LW}
 implies that the CW complex constructed is countable.
 \end{Exa}

\subsection*{Nilpotence}

\begin{Pro} \label{Gamma-nilpotent}
With respect to the hypotheses of Proposition \ref{Pro:milnor}, assume additionally that
$X$ has the homotopy type of a CW complex and that $E$ is a connected nilpotent space. Then each component of 
$\Gamma(\mathcal D)$ is nilpotent.
\end{Pro}

\begin{proof} Consider the commutative diagram
$$
\xymatrix{
 F(X,E;g) \ar[r] \ar[d] & F(X,B;p\circ g) \ar[d] \\
 F(X,E) \ar[d] \ar[r] & F(X,B) \ar[d] \\
F(A,E) \ar[r] & F(A,B)
}
$$
where the horizontal maps are all fibrations. By \cite[th.\ 2.5]{HMR},
each component of $F(X,E)$ is nilpotent.
It follows that each component of $F(X,E;g)$ is nilpotent by
\cite[th.\ 2.2]{HMR} applied to the left column of the diagram.
After restricting the fibration on the top
line to connected components, it follows again by \cite[th.\ 2.2]{HMR}
that each component of $\Gamma(\mathcal D)$ is nilpotent, since the latter
is a fibre by Proposition \ref{Pro:pullback section sequence}.
\end{proof}

\subsection*{Fibrewise Groups}

\begin{Def} A fibration $p\: E \to B$ is said to be a
{\it fibrewise group}
if it comes equipped with a map $m\:E\times_B E \to E$, a map $i\: E \to E$ and a section
$e\: B \to E$, all compatible with projection to $B$, such that
\begin{itemize}
\item $m$ is associative,
\item $e$ is a two sided unit for $m$,
\item $i$ is an inverse for $m$ and $e$
\end{itemize}
(In other words, $(p,m,e)$ defines a group object in the category of spaces over $B$.)
\end{Def}

If $(p,m,e)$ is a fibrewise group, then the space of sections $\Gamma(p)$ comes equipped
with the structure of a topological group with multiplication defined by pointwise multiplication
of sections.

Here is a recipe for producing fibrewise groups. Suppose $\zeta\:T \to X$ is a principal
$G$-bundle and $F$ is a topological group such that $G$ acts on $F$ through
homomorphisms (this means that one has a homomorphism $G \to \text{aut}(F)$, where
$\text{aut}(F)$ is the topological group consisting of  topological group automorphisms of $F$).
Then the associated fibre bundle
$$
T\times_G F \to X
$$
is easily observed to be a fibrewise group.

An important special case occurs when $F$ is $G^\ad$. We infer  that
$\Ad(\zeta)$ is a fibrewise group and so the gauge group $\GG(\zeta)$
has the structure of a topological group.
Similarly, when $F$ is $G^\ad$ and we let $PG$ act by conjugation, we
infer that  $\mathcal P(\zeta)$ has the structure of a
 topological group.

\begin{Exa}  \label{thm:UA proj gauge}
Returning to the $C^\ast$-algebra setting, we now show that the group
$UA_\zeta$ in the Introduction corresponds to a projective gauge group.

Let $X$ be a compact space and let
\[
\zeta \:T  \to  X
\]
be a principal $PU_n$-bundle
over $X$ with associated $C^\ast $-algebra $A_\zeta $. Then there is an natural
isomorphism of topological groups
\[
UA_\zeta  \,\,\cong\,\, \mathcal{P}(\zeta ). \]

The proof is as follows:
passing from $M_n$ to the subspace $U_n$ of unitaries in
each fibre of $\zeta$ yields a bundle
\[
U\zeta \: T  \times _{PU_n} U_n \to  X\, .
\]
The sections of this bundle are exactly $UA_\zeta $ but it is immediate
that the bundle itself is the bundle $\Pad(\zeta )$.
\end{Exa}

\begin{Rem}
Note that if $f: Y \to X$ is continuous then $f$ induces a map of unital $C^\star $-algebras $f^* : A_\zeta \to A_{f^*\zeta } $ 
and
this restricts to a homomorphism of unitary groups $U(A_\zeta )\longrightarrow U(A_{f^*(\zeta )} )$.  The naturality in the 
result
above is with respect to these maps.
\end{Rem}

\section{Rationalization of topological groups}
\label{sec:groups}
Now suppose that $G$  is a topological group having the homotopy type of
a CW complex.
As rationalization  commutes with products only up to homotopy,
the rationalization of the product structure
gives a map
$$
G_\QQ \times G_\QQ \to G_\QQ
$$
which may fail to be a group structure. It is, however, a group-like $H$-space.

The map $G \to G_\QQ$ is a homomorphism of $H$-spaces in the sense that the diagram
$$
\xymatrix{
G \times G \ar[r] \ar[d] & G \ar[d] \\
G_\QQ\times G_\QQ \ar[r] & G_\QQ
}
$$
commutes up to homotopy (and the homomorphism is compatible with
homotopy associativity).

\subsection*{Homotopy commutativity and rationalization}

\begin{Def} A group-like $H$-space $X$ is {\it homotopy commutative} if the
commutator map $[\,\,,\,\,]\:X \times X \to X$ is null homotopic.
It is said to be {\it rationally homotopy commutative}
if its rationalization $X_\QQ$ is homotopy commutative.
\end{Def}

The commutator map  induces an operation
on homotopy groups called the {\it Samelson product}. After tensoring with
the rationals, one obtains a graded Lie algebra structure (\cite[chap.\ X.5]{Wh}).

\begin{Def}
A homomorphism $X\to Y$ of connected $H$-spaces of CW type is said to be
{\it rational $H$-equivalence} if its rationalization $X_\QQ \to Y_\QQ$
is a homotopy equivalence.
\end{Def}

\begin{Pro}[\bf Scheerer {\cite[cor.\ 1]{Sch}}\rm ] \label{Pro:Scheerer}
Let $X$ and $Y$ be  connected group-like $H$-spaces having the homotopy type of
a CW complex. Then there is a rational $H$-equivalence
$$
X_\QQ \,\, \simeq \,\, Y_\QQ
$$
if and only if  there is an isomorphism of Samelson Lie algebras
$$
\left( \pi_*(X) \otimes \QQ, [ \, , \, ]\right) \,\, \cong  \,\,
\left( \pi_*(Y) \otimes \QQ, [ \, , \, ]\right)\, .
$$
\end{Pro}

We observe that the rationalization of such a space $X$ has the homotopy type
of a generalized Eilenberg Mac\,Lane space:
\begin{equation} \label{eq:G_Q} X_\QQ \simeq \prod_{j \geq 1} K(\pi_j(X) \otimes \QQ, j). \end{equation}
However, as pointed out in the introduction,
the multiplication on  $X_\QQ$
need not correspond to the standard multiplication.
We may detect when this
identification is multiplicative  in several ways.

\begin{Pro}[\bf cf.\ {\cite[th.\ 4.25]{LPSS}}\rm ] \label{Pro:H-abelian}  Let $X$ be a homotopy associative $H$-space having the 
homotopy type of a connected CW complex.  Then the following are equivalent:
\begin{itemize}
\item[(a)]  There is a homotopy equivalence
$$X_\QQ   \,\, \overset\sim\to \,\, ( \prod_{j \geq 1} K(\pi_j(X) \otimes \QQ, j) $$
which is also an $H$-map, where  the target has the standard multiplication.
\item[(b)]  The commutator map $X_\QQ \times X_\QQ \to X_\QQ$ is null homotopic.
\item[(c)] The Samelson Lie algebra $(\pi_*(X) \otimes \QQ, [ \, , \,  ])$ is abelian;
i.e., $[ \, , \,  ] = 0.$
\end{itemize}
\end{Pro}

\begin{Cor}\label{Cor:H-abelian}
Suppose $G$ is a connected topological group such that $BG$ has the rational
homotopy type of a loop space, i.e. there is a based space $Y$ and a rational homotopy
equivalence $BG \simeq_\QQ \Omega Y$. Then $G_\QQ$ is a homotopy commutative $H$-space
and is homotopy equivalent to a product of Eilenberg-Mac\,Lane spaces with standard multiplication.
\end{Cor}

\begin{Exa}
A topological group $G$ is said to satisfy {\it rational Bott periodicity}
if there is rational homotopy equivalence
$BG\simeq_\QQ \Omega^j G$ for some $j >0$.
Any such $G$ satisfies Corollary \ref{Cor:H-abelian}, and
is consequently rationally homotopy commutative.
In particular, the infinite unitary group $U_\infty$ is rationally homotopy commutative
(cf.\ Remark \ref{about_C*}).
\end{Exa}

\begin{Exa}
Suppose that the topological group $G$ is a direct limit $\Dirlim_n G_n$,
where each $G_n$ is rationally homotopy commutative. Then $G$ is rationally
homotopy commutative. This gives a second proof that $U_\infty$ is
rationally homotopy commutative (cf.\ Remark \ref{about_C*}).
\end{Exa}

 We now discuss the various hypotheses on the group $G$
in  our main results.
Recall from   Theorem \ref{bigthm:mainC}
we require that $G$ be a topological
group of  the homotopy type of a finite complex.
We first observe that
this class includes the connected Lie groups.

\begin{Lem} \label{Pro:G good} Every connected  Lie group $G$ has the homotopy
type of a finite CW complex.   \end{Lem}

\begin{proof}   By  \cite[th.\ A.1.2]{Wh},  $G$ has a maximal compact
subgroup $K$, unique up to conjugacy, such  that the inclusion $K \subset G$
is a homotopy equivalence. Then $K$, being a compact Lie group, has  the homotopy
type of a finite CW complex.
 \end{proof}

 We will make use of the following results whose proofs are classical.

\begin{Pro}\label{Pro:G properties} Suppose that $G$ is a connected, topological group
having the homotopy type of a finite CW complex. Then the following are true:
\begin{enumerate}
\item[(a)] The commutator map $G_\QQ \times G_\QQ \to G_\QQ$ is null homotopic.
\item[(b)] $\pi _2(G)$ is a  finite group.
\item[(c)] The classifying space  $BG$ has the rational homotopy type of a generalized Eilenberg-Mac\,Lane space and in
    particular is rationally homotopy equivalent to a loop space.
\end{enumerate}
\end{Pro}

\begin{proof}
The basic results of Milnor and Moore \cite{MM} on the structure of
Hopf algebras of characteristic zero imply $H^*(G; \QQ)$ is an  exterior algebra on a finite number of
odd degree generators. It follows that $H^*(BG;\QQ)$ is a polynomial algebra on
a finite number of generators of even degree. Represent each generator by a map
$x_i\: BG \to K(\QQ,n_i)$. Then the product map $f = \prod x_i \: BG \to \prod_i K(\QQ ,n_i)$
gives an isomorphism on rational homotopy groups. This proves (c). Applying the loop space
functor to the map $f$ gives (a) by Proposition \ref{Pro:H-abelian}.

To compute $\pi _2$ we may pass to universal covers
and hence assume that the groups are simply connected. Then
$H^*(G;\QQ ) = 0$ in degrees 1 and 2. This implies that
$H_2(G;\ZZ )$
is a finite group. The Hurewicz map $\pi _2(G) \to H_2(G;\ZZ ) $
is an isomorphism, so $\pi _2(G)$ is a finite group.
\end{proof}

\begin{Cor} \label{Cor:G properties} If $G$ is a connected topological group
having the homotopy type of a finite CW complex, then $G_{\QQ}$ is a homotopy commutative
$H$-space. In particular, $G_\QQ$ is homotopy equivalent as an $H$-space
to a product of Eilenberg-Mac\,Lane spaces with standard multiplication.
\end{Cor}

Finally, in Theorem \ref{bigthm:mainD} we restrict to the class of compact Lie groups $G$.
This restriction is chosen to govern the rational homotopy theory of $PG$ as we
explain now.  First we have the following general fact.

\begin{Lem} \label{Pro:PG properties} Let $G$ be a  connected Lie group. Then
\begin{enumerate}
 \item[(a)] $PG$ is a connected Lie group;
 \item[(b)] the inclusion $Z(G) \to G$
induces a monomorphism 
\[
\pi_1(Z(G)) \otimes \QQ  \to \pi_1(G) \otimes \QQ.
\]
\end{enumerate}
\end{Lem}

\begin{proof} $PG$ is a Lie group and it is connected since it is a quotient space of $G$.
By Proposition \ref{Pro:G properties},  $\pi _2(PG)$ is a finite group which implies the second statement.
\end{proof}

The preceding results imply in particular that  $G$ and $PG \times Z(G)$ have the same
homotopy type after rationalization.
Note, however, that there is no obvious map in either direction. We need to sharpen
this identification for our proof of Theorem \ref{bigthm:mainD}.
The following   classical fact due to E. Cartan explains our restriction there to compact Lie groups.

\begin{Pro} \label{Pro:G_0} Let $G$ be a compact,  connected Lie group. Then  there is  a    compact, connected Lie group $G_0$
such that the following hold.
\begin{enumerate}
\item[(a)] There is  a homomorphism $q \: G_0  \to G$ which is a rational homotopy equivalence.
\item[(b)] There is a splitting $G_0 \cong P(G_0) \times Z(G_0)$
\item[(c)]  The map $q$ carries $Z(G_0)$ to $Z(G)$ and induces an  isomorphism $P(G_0) \cong  PG.$  \end{enumerate}
\end{Pro}
\begin{proof}
By \cite[th.\ A.1.1]{Wh}, $G$ has a finitely-sheeted covering group $q \: G_0 \to G$ with
$G_0 = T^\ell \times  G'$, where  $G'$ is a simply connected compact Lie group with trivial center and
$T^\ell$ is the product of $\ell$-copies of $S^1$.
The results follow directly.
\end{proof}

%
%
\section{Preliminary results: finite complexes} \label{sec:prelims}
In this section, we prove Theorems \ref{bigthm:mainB}, \ref{bigthm:mainC} and \ref{bigthm:mainD}
 when $X$ is a finite CW complex.
For Theorem \ref{bigthm:mainB}, the result is a direct consequence of classical work of Thom and a localization theorem for 
function spaces due to Hilton, Mislin and Roitberg.
The proof of Theorem \ref{bigthm:mainC} makes use of Gottlieb's identity \cite[th.\ 1]{Got} for the gauge group in addition to 
the previous ingredients. We deduce Theorem \ref{bigthm:mainD}
from  Theorem \ref{bigthm:mainC} and Proposition \ref{Pro:G_0}.

First, we  have the famous result of H. Hopf, generalized by Thom \cite{Thom}:
\begin{equation} \label{eq:Thom}
\pi_q(F(X, K(\pi, p))) \,\, = \,\,H^{p-q}(X; \pi)\, .
\end{equation}
Next we have the results \cite[th.\ II.3.11, cor.\ II.2.6]{HMR}
of Hilton-Mislin-Roitberg on the function space $F(X, Y)$:

\begin{Pro}[\bf Hilton-Mislin-Roitberg {\cite{HMR}}\rm ] \label{Pro:HMR}
Let  $X$ be  a finite CW complex and $Y$ be a connected nilpotent space.
Then the induced map
 $$
F(X, Y) \to F(X, Y_\QQ)
$$
is a rationalization map on connected components.
\end{Pro}

{\flushleft (}More precisely, for any map $h\: X \to Y$, the map $F(X,Y)_{(h)} \to
F(X,Y)_{(\ell_Y\circ h)}$ is a rationalization map.)

\begin{Pro}[cf.\ {\cite[th.\ 4.28]{LPSS}}]  \label{Pro:Hnil}
Let  $G$ be a topological group having the homotopy type
of a finite connected CW complex.
Let $X$ be a finite CW complex.  Then the rationalization of
$F(X, G)\idc$ is a homotopy commutative $H$-space.
\end{Pro}
\begin{proof} Recall that $F(X, G)\idc$ is the component of $F(X,G)$
containing  the constant map. By  Corollary  \ref{Cor:G properties}, $G_{\QQ}$ is homotopy commutative.
It follows that $F(X,G_{\QQ})\idc$ is also homotopy commutative.
The result now follows from observing that
$F(X,G_{\QQ})\idc$ is the rationalization of $F(X,G)\idc$ by Proposition \ref{Pro:HMR}.
\end{proof}

 \begin{Rem}  The preceding result holds in greater generality: for
  $F(X, G)$ to be rationally homotopy commutative, one only
needs to assume that G is a rationally homotopy
commutative  $H$-space (see \cite[th.\ 4.10]{LPSS}).
\end{Rem}

We can now prove Theorem \ref{bigthm:mainB}    under the assumption
that  $X$ is a finite complex as in \cite[th.\ 4.28]{LPSS}.

\begin{Thm}[\bf Preliminary version of Theorem \ref{bigthm:mainB}\rm ] \label{thm:mainB}
Let $X$ be a finite CW complex and
let  $G$ be a connected topological group having the homotopy type of a finite CW complex.
Then
\begin{equation}
\label{eq:iso2}
\pi_*(F(X, G)\idc) \otimes \QQ \,\, \cong \,\, H^*(X; \QQ)
\, \widetilde{\otimes} \, \left( \pi_*(G) \otimes \QQ \right)\, .
\end{equation}
Furthermore, $F(X, G)\idc$
is rationally $H$-equivalent to a product of Eilenberg-Mac\,Lane spaces with
the standard loop multiplication, with degrees and dimensions corresponding
to  (\ref{eq:iso2}).
\end{Thm}
\begin{proof}  Using Proposition \ref{Pro:HMR},  we have
$$\pi_*(F(X, G)\idc) \otimes \QQ \cong \pi_*(F(X, G_\QQ)\idc)$$
and we may compute the latter using the identification (\ref{eq:G_Q})
and Thom's formula (\ref{eq:Thom}).
The identification of the rational $H$-type of  $F(X, G)\idc$
 is a direct consequence of Proposition \ref{Pro:Hnil}.
\end{proof}

\begin{Rem} \label{dispense_finiteness} Since the identification
(\ref{eq:G_Q}) only requires a group-like $H$-space,
the homotopy group calculation holds when $G$ is group-like.

The assumption that $G$ is homotopy finite is used   only to
conclude that $G_\QQ$ is homotopy commutative via Proposition \ref{Pro:Hnil}.
Consequently, the second conclusion of
Theorem \ref{thm:mainB} holds when $G$ is a group-like $H$-space such that
$G_\QQ$ is homotopy commutative.
\end{Rem}

We next  prove  Theorem \ref{bigthm:mainC} for $X$ a finite CW complex.

\begin{Thm}[\bf Preliminary version of Theorem \ref{bigthm:mainC}\rm ] \label{thm:mainC}
Let $G$ be a topological group having  the homotopy type of a finite connected
CW complex.
Let $\zeta$ be a principal $G$-bundle over a finite CW complex $X$.  Then there is
a  rational $H$-equivalence
$$\GG(\zeta)\idc \,\, \simeq_\QQ \,\, F(X, G)\idc\, .$$
\end{Thm}
\begin{proof}
By     \cite[th.\ 1]{Got} (see also Corollary \ref{cor:LG} below),
 there is a homotopy equivalence of $H$-spaces
$$
\mathcal{G}(\zeta) \,\, \simeq \,\, \Omega_h F(X, BG)\, ,
$$
where the space on the right is the based loop space of the function
space $F(X,BG)$ with basepoint given by
the classifying map $h\:X\to BG$ for the bundle $\zeta$.
Recall that  $F(X,BG)_{(h)}$ denotes the path component of $F(X,BG)$ containing $h$.
By   Proposition \ref{Pro:HMR} above,
$$
F(X, BG)_{(h)} \to F(X, (BG)_\QQ)_{(h')}
$$
is a rationalization map, where
$h'$ is $\ell_{BG}\circ h$, where $\ell_{BG}\: BG \to (BG)_\QQ$ is the rationalization map.
In particular, the displayed map is a rational homotopy equivalence.

By Corollary \ref{Cor:G properties},  $(BG)_\QQ$
has the homotopy type of a loop space. It follows that
all the components of $F(X, (BG)_\QQ)$ have the same homotopy type.
Consequently, there is a rational homotopy equivalence
$$
F(X, BG)_{(h)}  \simeq_\QQ F(X,(BG)_\QQ)\idc \, .
$$
The result now follows by taking the based loop space of both sides.
\end{proof}

\begin{Rem} \label{dispenseC}
It is clear from our  proof that the finiteness assumption on $G$ was used only to
conclude that $(BG)_\QQ$ has the homotopy type of a loop space.
In fact, one sees that the above argument works, without the finiteness assumption
on $G$, at the expense of assuming that $(BG)_\QQ$ has the structure of
 a group-like $H$-space.

 While the   identity
$$
\mathcal{G}(\zeta) \simeq \Omega_h F(X, BG)
$$
does extend to more general spaces $X$ (see Corollary \ref{cor:LG})), the
method of proof above is limited by  Proposition \ref{Pro:HMR}. Both the
nilpotence result \cite[th.\ 2.5]{HMR}
and the localization result Proposition \ref{Pro:HMR}
for $F(X, Y)$ require $X$ to be a finite CW complex.  In Section \ref{sec:HMR},  we extend the localization result to $X$ 
compact metric assuming nilpotence.
However, the nilpotence of the components of
$F(X, Y)$ is not expected to hold for general $X$.
\end{Rem}

The last goal of this section is to prove  Theorem \ref{bigthm:mainD} when  $X$ is
a finite complex.

\begin{Thm}[\bf Preliminary version of Theorem \ref{bigthm:mainD}\rm ]
\label{thm:mainD}
Let $G$ be a compact connected Lie group, and
let $\zeta\: T \to X$ be a principal $PG$-bundle over a finite CW complex $X$.  Then there is
a  rational homotopy equivalence of $H$-spaces
$$\mathcal P(\zeta)\idc \,\, \simeq_\QQ \,\, F(X, G)\idc.$$
\end{Thm}

\begin{proof} Case 1. Suppose first that $G$ splits as $Z(G) \times P(G)$. Then one has
an isomorphism of bundles over $X$ with total spaces
$$
(T \times Z(G)) \times_G G^\ad \,\, \cong  \,\, T \times_{PG} G^\ad \, .
$$
The result follows now by Theorem \ref{thm:mainC} applied to the bundle on the left.
\smallskip

{\flushleft Case 2.} This is the general case. By Proposition \ref{Pro:G_0}, there is a
compact Lie group $G_0$ and a homomorphism $q\:G_0 \to G$ which is a rational homotopy equivalence. Furthermore, this 
homomorphism induces an isomorphism $P(G_0) \cong P(G)$ and one
also has a splitting $G_0 \cong Z(G_0) \times P(G_0)$.

By Proposition \ref{U_0->U} below, the evident map
$$
Q_0 :=  EPG \times_{PG} G_0^{\ad} \to EPG \times_{PG} G^{\ad} := Q
$$
of fibrewise groups is a rational homotopy
equivalence of  nilpotent spaces.
Let $h \: X\to BPG$ classify the bundle with total space
$T\times_{PG} G^\ad$; then $h$ also classifies the bundle
with total space $T\times_{PG} G_0^\ad$.

Denote the lifting problem
$$
\xymatrix{
& Q \ar[d]\\
X \ar@{..>}[ur] \ar[r]_h & BPG
}
$$
by $\mathcal D$. Then the space of lifts $\Gamma(\mathcal{D})$
is the projective gauge group $\mathcal P(\zeta)$.
Denote the corresponding lifting problem with $Q$ replaced by $Q_0$ by $\mathcal D_0$.
Then the map $$\Gamma(\mathcal D_0) \to \Gamma(\mathcal D)$$
induced by the homomorphism $q\: G_0 \to G$ is both a rational homotopy equivalence
on components and
a map of $H$-spaces
(this is by a straightforward induction using the cell structure for $X$).
Similarly, $q$ induces a rational homotopy equivalence of $H$-spaces
$F(X,G_0) \to F(X,G)$.

By Case 1, we also have a rational equivalence
of $H$-spaces
$$
\Gamma(\mathcal D_0)\idc \,\, \simeq \,\, F(X,G_0)\idc \, .
$$
Assembling these three equivalences completes the proof.
\end{proof}

The remainder of this section is devoted to proving Proposition \ref{U_0->U} used in the proof above. We need some 
preliminary lemmas.

\begin{Lem} Assume $G$ is a compact connected Lie group.
Let $E = EG \times_G G^\ad$ and $Q = EPG\times_{PG} G^\ad$.
Then the map
$$
E \to Q
$$
induces a surjection on homotopy groups in each degree.
\end{Lem}

\begin{proof}
 There is a homotopy fibre sequence
 $$
 BZ(G) \to E \to Q \, .
$$
Taking the long exact homotopy sequence, we
infer that $\pi_*(E) \to \pi_*(Q)$ is an isomorphism
when $\ast \ne 3$ (here we are using the fact  $Z(G)$ is a torus).
Consequently, we have an exact sequence
$$
0 \to \pi_3(E) \to  \pi_3(Q) \to \mathbb Z^\ell \to \pi_2(E)\to \pi_2(Q) \to 0
$$
where $\ell =$ rank of $Z(G)$.

We can calculate $\pi_3(E)$ using the long exact sequence of the
 fibration $E \to BG$; one sees
(using the fact that $\pi_2(G) = 0$) that it is isomorphic to $\pi_3(G)$.
Likewise, we see that $\pi_3(Q)$ is also isomorphic to $\pi_3(G)$ and the homomorphism
$\pi_3(E) \to \pi_3(Q)$ is in fact an isomorphism.
\end{proof}

\begin{Lem} If $E \to B$ is a fibration of connected spaces having the
homotopy type of a CW complex. Assume $\pi_*(E) \to \pi_*(B)$
is surjective in every degree and $E$ is nilpotent. Then $B$ is
nilpotent.
\end{Lem}

\begin{proof} The quotient of a nilpotent group is again nilpotent, so
$\pi_1(B)$ is nilpotent. Furthermore, when $k \ge 2$,
we have a short exact sequence of $\pi_1(E)$ modules
$$
0 \to \pi_k(F) \to \pi_k(E) \to \pi_k(B) \to 0
$$
where $F$ denotes the fibre at the basepoint.
The nilpotency of the middle module guarantees that $\pi_k(B)$ is
also a nilpotent $\pi_1(E)$-module
 (see \cite[prop.\ 4.3]{HMR}). This module structure
arises from the homomorphism $\pi_1(E) \to \pi_1(B)$ by restriction.
Since this homomorphism is surjective, it follows that $\pi_k(B)$ is
a nilpotent $\pi_1(B)$-module.
\end{proof}

\begin{Lem} \label{Unilpotent} Let $G$ be a compact connected Lie group. Then the space
$Q = EP(G) \times_{P(G)} G^\ad$  is nilpotent.
\end{Lem}

\begin{proof} There is a homotopy fibre sequence
$$
BZ(G) \to E \to Q\, ,
$$
where $E = EG\times_G G^\ad$. Then $E$ is homotopy equivalent to $LBG$, the free loop
space of $BG$ (cf.\ Lemma \ref{lem:LG}). We infer that $E$ is nilpotent by \cite[th.\ 2.5]{HMR}.
Now apply the preceding lemmas.
\end{proof}

\begin{Pro} \label{U_0->U} Let $G$ be a compact connected Lie group and
let $q\:G_0 \to G$  be as in Proposition \ref{Pro:G_0}. Then the map
of fibrewise groups
$$
Q_0 := EP(G) \times_{P(G)} G_0^\ad \to  EP(G) \times_{P(G)} G^\ad =: Q
$$
is a rational homotopy equivalence of nilpotent spaces.
\end{Pro}

\begin{proof} Both $Q$ and $Q_0$ are nilpotent by Lemma \ref{Unilpotent}.
By applying rationalization to the diagram
$$
\xymatrix{
G_0^\ad \ar[r] \ar[d]_q & Q_0 \ar[r]\ar[d] & BP(G)\ar@{=}[d] \\
G^\ad \ar[r] & Q \ar[r] & BP(G)
}
$$
whose rows are fibre sequences,
and using the fact that rationalization  preserves fibrations
(\cite[th.\ 3.12]{HMR}) we infer that the map $Q_0 \to Q$ is a rational
homotopy equivalence.
\end{proof}


\section{Limits and function spaces}\label{sec:functions}

  When $X$ is a compact metric space, a classical result of Eilenberg and Steenrod
  \cite[th.\ X.10.1]{ES} gives an inverse system
of finite simplicial (CW) complexes  $X_j$  and compatible maps $h_{j} \: X  \to X_j$
  such that the induced map
  $$h\: X    \to    \Invlim_{j}X_j $$
  is a homeomorphism.  This  result and its generalization
  are at the core of our method for passing  from finite complexes
  to compact metric spaces.

In this and subsequent sections, we consider  both direct and inverse limits.
Suppose
$\{ X_j, p_{ij} \}$ is an inverse system of spaces,
where $p_{ij} \: X_i \to X_j$ are maps, $j \leq i$.
Given compatible maps $h_j \: X \to X_j$, one has an induced map $h = \Invlim_j h_j \: X \to \Invlim_{j} X_j.$

 We record the following basic result.

\begin{Pro}[\bf Eilenberg-Steenrod {\cite[th.\ X.10.1, X.11.9]{ES}}\rm ]
 \label{Pro:esinvlim}  Let $X$ be a compact Hausdorff space.
\begin{enumerate} \item[(a)]
There exists an inverse system of finite CW complexes $\{X_j, p_{ij} \}  $
 and compatible maps $h_j \: X  \to X_j$ inducing a homeomorphism
$$
h = \Invlim_{j} h_j \: X \to \Invlim_{j}\,  X_j \, .
$$
\item[(b)]
Given a map $f\:X\to Y$ in which $Y$ is a CW complex,
there exists an index \,$m$
and a cellular map
$f_m\: X_m\to Y$
 such that the composite
$$
\begin{CD}
X @> h _m >> X_m @> f_m >>
Y
\end{CD}
$$
is homotopic to $f$.
\end{enumerate}
\end{Pro}

\begin{Pro} \label{lim-onto} Under the hypotheses of Proposition \ref{Pro:esinvlim}, the map of sets
$$
\Dirlim_j\,  [X_j, Y] \to  [X,Y]
$$
is a bijection.
\end{Pro}

\begin{proof} Surjectivity is a direct consequence of
Proposition \ref{Pro:esinvlim} (b).
Injectivity
 is a consequence of Spanier's method of proof of \cite[th.\ 13.4]{Span}.
 In Spanier's case, $Y = S^n$ and $X$ has Lebesgue covering
 dimension at most $2n-2$ and his limit is taken in the category of abelian groups.
 However, Spanier remarks that the dimension condition can be dropped provided that the
 limit is taken in the category of sets \cite[p.\ 228]{Span}. Furthermore, an
 inspection of his proof shows that it generalizes without change
to $Y$ an arbitrary finite simplicial complex. The argument is then
completed by recalling that any finite CW complex has the homotopy type of a simplicial
 complex.
\end{proof}

We will need to extend this proposition to a certain class of pairs.
Suppose now that $(X,A)$ is a pair, where $X$ is a compact Hausdorff space and $A\subset
X$ is a closed cofibration. We assume that $(X,A)$ is expressed as an inverse limit
of pairs $(X_j,A_j)$ where the latter is a finite CW pair.  Such a decomposition
exists by the relative version of \cite[Ch.\ X, th.\ 10.1, 11.9]{ES}. As above, write $p_{ij} \colon X_i \to X_j$ for   $j \leq i$ and 
$h_j \colon X \to X_j$ for the  structure maps. We use the same notation for the restrictions of these maps to $A_j$ and to 
$A$, respectively.  Let
$Y$ be a CW complex, and suppose that one is given a fixed map
$g_m\: A_m \to Y$ for some $m$ and
define $g_j\: A_j \to Y$ for $j > m$  by $g_m\circ p_{jm}$.
Define $g$ to be the composite $g_m\circ h_m$. Let
$$
[X,Y;g]
$$
denote the set of homotopy classes of maps $X \to Y$ which coincide with
 $g$ on the subspace $A$ (where homotopies are
 required to be constant on $A$). Similarly, we have $[X_j,Y;g_j]$ and a map
of sets
$$
[X_j,Y;g_j]\to [X,Y;g]
$$
(for $j \ge m$) which is compatible with the index $j$.

\begin{Lem} \label{technical} Assume there are compatible retractions
$r_j\: X_j \to A_j$ inducing a retraction $r\: X \to A$.
Then the map
$$
\Dirlim_j \, [X_j,Y;g_j]\to [X,Y;g]
$$
is a bijection.
\end{Lem}

\begin{proof} Let $i\: A \to X$ and $i_j\: A_j\to X_j$ be the inclusions,
and let $u\: [X,Y;g] \to [X,Y]$ and $u_j\: [X_j,Y;g_j] \to [X,Y]$
be the evident maps.
For each $j$, one has a commutative diagram of sets
$$
\xymatrix{
[X_j,Y;g_j] \ar[r] \ar[d]_{u_j} & [X,Y;g] \ar[d]^{u}\\
[X_j,Y] \ar[r]^{h_j^*} \ar[d]_{i_j^*} & [X,Y]\ar[d]^{i^*} \\
[A_j,Y] \ar[r] & [A,Y]
}
$$
where the bottom terms are pointed sets. Furthermore, if $r\: X \to A$ is
a retraction, then $g\circ r$ is a basepoint for $[X,Y]$ making
$i^*$ into  a split surjection of based sets. The right column
is in fact the tail-end of the long exact homotopy sequence of the fibration
$F(X,Y) \to F(A,Y)$, which is also equipped with section. It follows
from this observation that $u$ is one-to-one. Similarly $u_j$ is one-to-one.

Taking direct limits results in a diagram such that middle and bottom maps
are isomorphisms. The rest of the argument follows from an elementary diagram chase,
using the fact that $u_j$ and $u$ are one-to-one (we leave the details to the reader).
\end{proof}

Now, let $f_m\: X_m \to Y$ be a fixed map and
define $f_j\: X_j \to Y$ for $j > m$  by $f_m\circ p_{jm}$.
Define
$f$ to be the composite $f_m\circ h_m$. Then the map of function spaces
$F(X_j,Y) \to F(X,Y)$ sends $f_j$ to $f$, so we have a map of based spaces that is
compatible with the inverse system.

\begin{Thm} \label{thm:Flim}
 The  inverse system of based spaces above induces an isomorphism of groups
 $$
 \Dirlim_j \pi_n(F(X_j, Y);f_j) \,\, \cong \,\,  \pi_n(F(X, Y);f)
 $$
in all degrees.
 \end{Thm}

\begin{proof} By \cite[prop.\ IX.2]{MacLane},
the limit of a direct system of (abelian) groups coincides with
the limit taken in the category of sets.
\smallskip

\flushleft{Case 1.} $n = 0$. This case is just a reformulation of
 Proposition \ref{lim-onto}.
\medskip

{\flushleft Case 2.} $n>0$.
Observe that
$$
[X\times S^n,Y;f] \,\, = \,\, \pi_n(F(X, Y);f)\, , $$
 where on the left we are taking homotopy classes of maps
 $X \times S^n \to Y$ which coincide
 with $f$ on $X\times \ast = X$ .
 Note that each inclusion
$X_j\times \ast \subset X_j\times S^n$ is a retract, and these retractions
are compatible.
The result then follows from Lemma \ref{technical} with
$X\times S^n$ in place of $X$, $X\times \ast$ in place of $A$, $X_j \times S^n$
in place of $X_j$ and $X_j \times \ast$ in place of $A_j$.
 \end{proof}

\subsection*{Limits and section spaces}
Assume that  $(X,A) = \Invlim_j (X_j,A_j)$  as above, where each $(X_j,A_j)$ is
a finite CW pair. Suppose that for some index $m$ one is given a lifting problem
$$
\xymatrix{
A_m \ar[r]^{g_m} \ar[d]_{\cap} & E \ar[d]^p \\
X_m \ar[r]_{f_m} \ar@{..>}[ur] & B
}
$$
denoted $\mathcal D_m$. Here we assume that
$p\:E\to B$ is a fibration in which $E$ and $B$ have the homotopy type of CW complexes.
Using the maps $(X_j,A_j) \to (X_m,A_m)$, we obtain another lifting problem, denoted
$\mathcal D_j$. Then one has maps $\Gamma(\mathcal D_j) \to \Gamma(\mathcal D_{j+1})$ for $j \ge m$.
Let $\tilde f_m \: X_m \to E$ be any lift. Then we obtain
basepoints $\tilde f_j\in \mathcal D_m$ for $j \ge m$.

Let $f\: X \to B$ denote the composite of $h_m\circ f_m$ and similarly, let
$g\: A \to E$ be the composite $h_m\circ g_m$, where $h_m \: (X,A) \to (X_m,A_m)$
is the structure map. Then we get a lifting problem $\mathcal D$
$$
\xymatrix{
A \ar[r]^g \ar[d]_{\cap} & E \ar[d]^p \\
X \ar[r]_f \ar@{..>}[ur] & B \, .
}
$$
Let $\tilde f\: X \to E$ be the basepoint of $\mathcal D$ determined by $\tilde f_m$.

\begin{Thm} \label{thm:lim-section} The map of based sets
$$
\Dirlim_j \, \pi_n (\Gamma(\mathcal D_j);\tilde f_j) \to \pi_n (\Gamma(\mathcal D);\tilde f)
$$
is an isomorphism in every degree $n \ge 0$, where the direct limit is
taken in the category of sets.
\end{Thm}

\begin{proof}   For each $n$, one has a map of long exact homotopy sequences
$$
\xymatrix{
\cdots \ar[r]^<<<<<{\partial} & \pi_n(\Gamma(\mathcal D_j);\tilde f_j)\ar[r]^<<<<<{a_j} \ar[d]_c &
\pi_n(F(X_j,E);g_j) \ar[r]^{b_j} \ar[d]_d & \pi_n(F(X_j,B);f_j)\ar[r]\ar[d]^e & \cdots \\
\cdots \ar[r]_<<<<<{\partial} & \pi_n(\Gamma(\mathcal D);\tilde f)\ar[r]_<<<<<<{a}  &
\pi_n(F(X,E);g) \ar[r]_b  & \pi_n(F(X_j,B);f)\ar[r] & \cdots
}
$$
as given by Proposition \ref{Pro:pullback section sequence}.

To prove surjectivity, let $x\in \pi_n (\Gamma(\mathcal D);\tilde f)$ be any element.
By Theorem \ref{thm:Flim}, $a(x) = d(y)$ for some $y,$
provided that $j$ is  sufficiently large. Then $b_j(y)$
is trivial provided $j$ is large, again by  \ref{thm:Flim}.
It follows that $y = a_j(z)$ for some $z$. Then $a(c(z)-x) = 0$,
so $x = c(z) - \partial u$ for some $u$. If $j$ is large,
one has $u = e(u')$ for some $u'$. Consequently,
$x = c(z - \partial u')$. This establishes surjectivity.
A similar diagram chase, which we omit, gives injectivity.
\end{proof}

\section{Localization of function spaces revisited} \label{sec:HMR}

The purpose of this section is to extend the Hilton-Mislin-Roitberg localization result  (Proposition \ref{Pro:HMR}) for function 
spaces $F(X, Y)$
to the case $X$ compact metric and $Y$ nilpotent CW provided the particular function space component
 is known, a priori, to be nilpotent.

 Suppose that $X$ is a compact metric space and $X = \Invlim_j X_j$ as above, where
 each $X_j$ is a finite CW complex. Let $Y$ be a nilpotent space.
 Let $\ell_Y\: Y \to Y_\QQ$ be the rationalization map. Let
 $f\: X \to Y$ be a fixed map and consider the connected component
  $F(X,Y)_{(f)}$ of the function space.

 \begin{Thm} \label{thm:HMR extended} If $F(X,Y)_{(f)}$ is nilpotent,
  then the induced map
$$
F(X, Y)_{(f)} \to
F(X, Y_{\QQ})_{(\ell_Y\circ f)}
$$
is a rationalization map.
\end{Thm}

\begin{proof} By Proposition \ref{Pro:esinvlim}, we can assume without loss in generality that $f$ factors as
$X \to X_m \to Y$. Let $f_m \: X_m \to Y$ denote the factorizing map, and define
$f_j\: X_j \to Y$ for $j > m$ to be the composite $p_{jm}\circ f_m$, where
$p_{jm}\: X_j \to X_m$ is the structure map in the inverse system.
The approximation $X \cong \Invlim_j X_j$ gives
rise to a commutative diagram
$$
\xymatrix{
{\Dirlim}_j \pi_n(F(X_j,Y);f_j) \ar[r]^>>>>>>>>>{\cong} \ar[d] & \pi_n(F(X,Y);f) \ar[d]\\
\Dirlim_j \pi_n(F(X_j,Y_{\QQ});\ell_Y \circ f_j) \ar[r]_>>>>>{\cong}  & \pi_n(F(X,Y_{\QQ});\ell_Y \circ f) \\
}
$$
where the horizontal maps are bijections
by Theorem \ref{thm:Flim}. Apply the rationalization functor to the diagram and
use the fact that rationalization commutes with direct limits. This results in
a commutative diagram
$$
\xymatrix{
\Dirlim_j \pi_n(F(X_j,Y);f_j)_\QQ \ar[r]^>>>>>>>>>{\cong} \ar[d]_\cong & \pi_n(F(X,Y);f)_\QQ \ar[d]\\
\Dirlim_j \pi_n(F(X_j,Y_{\QQ});\ell_Y \circ f_j) \ar[r]_>>>>>>{\cong}  & \pi_n(F(X,Y_{\QQ});\ell_Y \circ f) \\
}
$$
where the left vertical map is an isomorphism by Proposition \ref{Pro:HMR}.
It follows that the right vertical map is an isomorphism as well.
\end{proof}


\section{Proof of the main results} \label{sec:sections}
We are now in a position to prove the main theorems in their complete generality.

\begin{proof}[{\bf Proof of Theorem \ref{bigthm:mainB}}]
Recall we are assuming $X$ is a compact metric space
and $G$ is a connected CW topological group having the homotopy type of a
finite complex. We need to establish an isomorphism
$$\pi_*(F(X, G)\idc) \otimes \QQ \,\, \cong  \,\, \Check{H}^*(X, \QQ)
\widetilde{\otimes} \big(\pi_*(G) \otimes \QQ \big)\, .$$
By Theorem \ref{thm:mainB}, the corresponding result
holds for $X_j$ a finite CW complex.
Write $X = \Invlim X_j$
as usual for
finite complexes $X_j$. Then for each $j$ we have a natural isomorphism
\[
\pi_*(F(X_j, G)\idc) \otimes \QQ \,\, \cong \,\,  {H}^*(X_j; \QQ)
\widetilde{\otimes} \big(\pi_*(G) \otimes \QQ \big)\, .
\]
 Take direct limits on both sides
and use the fact that
\[
\Dirlim (A_j \otimes B)\,\, \cong \,\,(\Dirlim A_j) \otimes B
\]
for abelian groups to obtain the isomorphism
\[
\big(\Dirlim \pi_*(F(X_j, G)\idc)\big) \otimes \QQ \,\,\cong \,\, \bigg( \Dirlim {H}^*(X_j; \QQ)\bigg)
\widetilde{\otimes} \big(\pi_*(G) \otimes \QQ \big)\,   .
\]
The continuity property of  \v{C}ech cohomology \cite[th.\ 12.1]{ES}
implies that
\[
 \Dirlim {H}^*(X_j; \QQ)\,\, \cong \,\, \Check {H}^*(X; \QQ)\,  .
\]
Then use   Theorem \ref{thm:Flim} to identify
\[
\Dirlim \pi_*(F(X_j, G)\idc) \,\, \cong \,\, \pi_*(F(X, G)\idc)
\]
 which gives the result at the level of homotopy groups.

  Finally,   use Theorem \ref{thm:HMR extended}
  to obtain  a homotopy equivalence of $H$-spaces
  $$(F(X, G)\idc )_\QQ \,\, \simeq  \,\, F(X, G_\QQ)\idc \, .$$
 The last part of Theorem \ref{bigthm:mainB} now follows from Propositions \ref{Pro:Hnil}
 and \ref{Pro:H-abelian}.
\end{proof}

\begin{proof}[{\bf Proof of Addendum \ref{addB}}]
As was observed in Remark \ref{dispense_finiteness}, the homotopy group
calculation of Theorem \ref{thm:mainB} holds when $G$ is a group-like H-space,
and the homotopy commutativity holds whenever $G$ is rationally homotopy commutative.
The above proof of Theorem \ref{bigthm:mainB}, which uses Theorem \ref{thm:mainB},
therefore holds in the stated generality.
\end{proof}

We refocus   on the case of the adjoint  and projective adjoint bundles.
As usual, let $X$ be a compact metric space, and write
$X = \Invlim_j X_j$ for an inverse system of finite complexes $X_j$.
 Let
 $$
 \zeta \: T \to X
 $$
  be the given principal $G$-bundle, where $G$ is of CW type. Let
  $f\: X \to BG$ be a classifying map for $\zeta$. By Proposition \ref{lim-onto},
  we can assume without loss in generality that $f$ factors as
 $$
 X \to X_m \overset{f_m}\to BG
 $$
 for some index $m$.
 For $j > m$, define $f_j \: X_j \to BG$ by taking the composite of $f_m$ with the map $X_j \to X_m$.
 This defines a principal $G$-bundle $\zeta_j\: T_j \to X_j$ for each $j \ge m$.

For each $j$ we have a lifting problem
  $$
\xymatrix{
& EG\times_G G^{\ad} \ar[d]^{\Ad(\zeta)} \\
X_j \ar[r]_{f_j} \ar@{..>}[ur] & BG
}
$$
whose space of sections is just the gauge group $\GG(\zeta_j)$. Furthermore,
one has a direct system of topological groups
$$
\GG(\zeta_m)\to \GG(\zeta_{m+1}) \to \cdots
$$
equipped with compatible homomorphisms $\GG(\zeta_j) \to \GG(\zeta)$.
 By Theorem \ref{thm:lim-section}, the homomorphism
$$
\Dirlim_j \, \pi_n (\GG(\zeta_j)\idc) \to \pi_n (\GG(\zeta)\idc)
$$
is an isomorphism for $n \ge 0$. A similar statement holds in the projective bundle
case. Summarizing, we obtain the following description of the homotopy groups of the
gauge group and of the projective gauge group.

\begin{Pro} \label{Cor:limGG}  Let $X$ be a compact metric space and suppose
$X = \Invlim_j X_j$ for an inverse system of finite complexes $X_j$.
Then, with notation as above,
   $$\pi_*(\GG(\zeta)\idc) \cong \Dirlim\pi_*( \GG(\zeta_j)\idc) \hbox{\, and \,}  \pi_*(\mathcal{P}(\zeta)\idc) \cong \Dirlim  
   \pi_*(\mathcal{P}(\zeta_j)\idc).
   $$
After rationalization, these become isomorphisms of rational Samelson algebras.
\end{Pro}

\begin{proof} The only thing we need to prove is the last statement. This follows
because the map inducing the isomorphism in each case is induced
from maps of $H$-spaces.
They thus induce isomorphisms of rational Samelson Lie algebras.
\end{proof}

\begin{proof}[{\bf Proof of Theorem \ref{bigthm:mainC}}]
Combining Proposition \ref{Cor:limGG} and the preliminary version of Theorem
\ref{bigthm:mainC} for finite complexes (Theorem \ref{thm:mainC}) with Theorem \ref{bigthm:mainB}  one sees   that    
$\GG(\zeta)\idc$ has rational homotopy groups  given
by Theorem \ref{bigthm:mainB}. Further, since a direct limit of abelian Lie algebras    is   abelian, we conclude 
$\GG(\zeta)\idc$ has abelian rational Samelson Lie algebra.
This, in turn,    implies
there exists an $H$-equivalence $\GG(\zeta)\idc \simeq_\QQ F(X, G)\idc$
by Proposition \ref{Pro:Scheerer}.
\end{proof}

\begin{proof}[{\bf Proof of Theorem \ref{bigthm:mainD}}]
The proof is similar to the preceding one. In this case, one combines
Proposition \ref{Cor:limGG} and the preliminary version of Theorem
\ref{bigthm:mainD} for finite complexes (Theorem \ref{thm:mainD}) to get that
$\mathcal P(\zeta)\idc$ has rational homotopy groups  given
by Theorem \ref{bigthm:mainB}. The rest of the argument is as in  the proof of Theorem
\ref{bigthm:mainC}.
\end{proof}

\begin{proof}[{\bf Proof of Addendum \ref{addC}}]
See Remark \ref{dispenseC}.
\end{proof}

\begin{proof}[{\bf Proof of Theorem \ref{bigthm:mainA}}]  The proof is a direct
consequence of Example \ref{thm:UA proj gauge},  Theorem \ref{bigthm:mainD}  for $G = U(n)$, and
the well-known result
  $$\pi_*(U(n)) \otimes \QQ \,\, \cong \,\, \QQ(s_1,   \ldots, s_n)$$
where $|s_i| = 2i-1$.
\end{proof}


\section{Appendix: on the free loop space}
In this section, we sketch  a proof of  ``Gottlieb's identity'' for the gauge group used in the proof of  Theorem \ref{thm:mainC}.   
While Gottlieb's original proof requires the base space $X$ of the given principal $G$-bundle to be a finite CW complex, our 
proof requires only that the bundle be a pullback of the universal principal $G$-bundle.

Given a space $X$,
let $LX = F(S^1, X)$ be its space of unbased loops.
Evaluating loops at their basepoints  gives a fibration
$LX \to X$.  For a topological group $G$ of CW type, let $\xi\: EG \to BG$
be the universal bundle, and let $\Ad(\xi)\: EG \times_G G^\ad \to BG$
be the associated adjoint bundle.
Then the following result is folklore.

\begin{Lem} \label{lem:LG} Let $G$ be any topological group of CW type. Then there is
a fibrewise homotopy equivalence
$$
    L(BG) \,\,  \simeq  \,\, EG \times_G G^{\ad}
    $$
of fibrewise $H$-spaces over $BG$.
\end{Lem}

\begin{proof} Let $G \times G$ act on $G^\ad$ by the rule $(g,h)\cdot x =
gxh^{-1}.$ Then the restriction of
this action to the image of the  diagonal $\Delta: G \to  G \times G$ coincides with the given action of $G$ on $G^{\ad}$. We 
have a pullback square
$$
\xymatrix{E(G \times G) \times_G G^{\ad} \ar[rr]^{E\Delta} \ar[d] &&
E(G\times G) \times_{(G\times G)} G^{\ad} \ar[d] \\
BG = E(G \times G)/G \ar[rr]_{B\Delta} && B(G \times G)}
$$
in which the vertical maps are fibrations and the
horizontal maps are induced by $\Delta.$
The space  $E(G\times G) \times_{(G\times G)} G^{\ad}$ may be identified
with $BG$.  To show this, we first quotient out by the action of the left-hand copy of
$G$ in $G \times G$. Since this action is free, we obtain $EG$. Thus
when we take the quotient by the right-hand copy of $G$ we get $EG/G =
BG.$  It follows that $EG \times_G G^{\ad} =  E(G \times G) \times_G G^{\ad}$ is identified with
the homotopy pullback of the diagonal of $BG$ with itself. But the latter
coincides with the actual pullback of the diagram
$$ \xymatrix{ & (BG)^I \ar[d]^{p} \\
BG \ar[r]_>>>>>{B \Delta} & BG \times BG }$$
where $(BG)^I = F(I, BG)$  is the free path space of $BG$, and $p$ is
the fibration which evaluates a path at its endpoints. This pullback
identically coincides with $L(BG)$.
\end{proof}

\begin{Cor}[\bf ``Gottlieb's Identity'' {\cite[th.\ 1]{Got}}\rm ] \label{cor:LG}
Let $G$ be any topological group of CW type.  Let $\zeta\: T \to X$ be a principal $G$-bundle  induced from the universal 
principal $G$-bundle by   a map $h_\zeta\: X \to BG$.
Then there is a homotopy equivalence of $H$-spaces
$$
\Gamma(\Ad(\zeta)) \,\, \simeq \,\, \Omega_{h_\zeta} F(X,BG)\, ,
$$
where the right side denotes the based loop space of $F(X,BG)$ with loops based at $h_\zeta.$

\end{Cor}

\begin{proof}  $\Gamma(\Ad(\zeta))$ coincides with  the space of solutions to the
lifting problem
$$
\xymatrix{
                          & EG\times_G G^\ad  \ar[d] \\
X \ar[r]_{h_\zeta} \ar@{..>}[ur] & BG \, .
}
$$
Denote this lifting problem by $\mathcal D$. Using Lemma \ref{lem:LG}, we see that
$\Gamma(\mathcal D)$ is homotopy equivalent as an $H$-space to the space of lifts
$$
\xymatrix{
                          & L(BG) \ar[d]\\
X \ar[r]_{h_\zeta} \ar@{..>}[ur] & BG \, .
}
$$
An unraveling of definitions shows that the latter
is the space of maps $X \times S^1 \to BG$ whose restriction
to $X \times \ast$ coincides with $h_\zeta$. But this is identical to the space
$\Omega_{h_\zeta} F(X,BG)$ by means of the exponential law.
\end{proof}




\end{document}